\documentclass{amsart}
\usepackage{amsmath,amsthm,amssymb,IMjournal,bm}

\begin{document}
\newtheorem{The}{Theorem}[section]
\newtheorem{lem}[The]{Lemma}
\newtheorem{defi}[The]{Definition}
\newtheorem{ex}[The]{Example}
\newtheorem{remark}[The]{Remark}
\newtheorem{cor}[The]{Corollary}

\numberwithin{equation}{section}

\title{Guaranteed two-sided bounds on all eigenvalues of preconditioned
diffusion and elasticity problems solved by the finite element method}

\author{\|Martin |Ladeck\' y|, Prague,
        \|Ivana |Pultarov\' a|, Prague, 
        \|Jan |Zeman|, Prague}

\rec {August, 2019}


\abstract 
A method of estimating all eigenvalues of a preconditioned discretized scalar diffusion operator with Dirichlet boundary conditions has been recently introduced in T.~Gergelits, K.-A.~Mardal, B.~F.~Nielsen, 
Z.~Strako\v{s}: Laplacian preconditioning of elliptic PDEs: Localization of the eigenvalues of the discretized operator, SIAM Journal on Numerical Analysis 57(3) (2019), 1369--1394. Motivated by this paper, we offer a slightly different approach that extends the previous results in some directions. Namely, we provide bounds on all increasingly ordered eigenvalues of a general diffusion or elasticity operator with tensor data, discretized with the conforming finite element method, preconditioned by the inverse of a matrix of the same operator with different data. Our results hold for mixed Dirichlet and Robin or periodic boundary conditions applied to the original and preconditioning problems. The bounds are two-sided, guaranteed, easily accessible, and depend solely on the material data.
\endabstract



\thanks
All authors acknowledge the financial support received from the Czech Science Foundation (project No. 17-04150J). This research has been performed in the Center of Advanced Applied Sciences (CAAS), financially supported by the European Regional Development Fund (project No.~CZ.02.1.01/0.0/0.0/16\_019/0000778).
 Ivana Pultarov\' a and Martin Ladeck\' y were supported also by the Grant Agency of the Czech Technical University in Prague, grant No.~SGS19/002/OHK1/1T/11.
\endthanks

\section{Introduction}
In 2009, Nielsen, Tveito, and Hackbusch studied 
in~\cite{NielsenTH2009} spectra of elliptic differential operators
of the type $\nabla\cdot k\nabla$ defined in 
infinite-dimensional spaces 
which are preconditioned using the inverse of the Laplacian.
They proved that the range of the scalar coefficient
$k$ is contained in the 
spectrum of the preconditioned operator, provided that $k$ is continuous. 
Ten years later, Gergelits, Mardal, Nielsen, and Strako\v s
showed in~\cite{G} without any assumption
about the continuity of the scalar function
$k$ that there exists a one-to-one pairing 
between the eigenvalues of the  
discretized operator of the type $\nabla\cdot k\nabla$ preconditioned by the inverse of the discretized Laplacian 
and the intervals determined by the images under $k$ 
of the supports of the conforming finite element (FE)
nodal basis functions used for the discretization.

The present paper contributes to the results of~\cite{G} and generalizes some of them.
While in~\cite{G}, a one-to-one pairing between the eigenvalues and
images of the scalar data $k$ defined on supports of the
FE basis function is proved, we introduce guaranteed 
two-sided bounds on all individual eigenvalues. 
Our approach is based on the Courant--Fischer
min-max principle. Similarly as in~\cite{G},
the bounds can be obtained solely from the data of
the original and preconditioning problems defined on supports of the
FE basis functions.
While in~\cite{NielsenTH2009} and~\cite{G} only 
the diffusion operator with scalar data is considered and
the Laplacian operator is used for preconditioning, we 
treat also the diffusion operator with 
tensor data and with Dirichlet or Robin boundary conditions
for both the original and preconditioning operators.
Our theory also applies to operators with non-zero null spaces
and to operators with vector valued unknown functions;
as an example we study the elasticity operator with
general tensor data.
Any kind of conforming FE basis functions can be employed
for discretization; the sets of the FE basis functions must be 
the same for the original and preconditioning operators.
For the sake of brevity, the name preconditioning matrix (operator)
will be used for the matrix $\widetilde{\mathsf{M}}$ (or operator)
which is (spectrally) 
close to the original matrix $\mathsf{M}$ (or operator, respectively) rather than for the inverse of $\widetilde{\mathsf{M}}$.
In contrast, in literature, including~\cite{G}, 
$\widetilde{\mathsf{M}}^{-1}$ is often called the preconditioning matrix.


For numerical solution of sparse discretized elliptic partial differential 
equations, the conjugate gradient method (or Krylov subspace methods, in general) is a method of choice; see, e.g., \cite{LiesenS,vanderVorst,SaadBook}. 
It is well known, that its convergence depends on
distribution (clustering) of eigenvalues of the related matrices
and on magnitudes of the associated eigenvectors within the initial residual. 
For example, large outlying eigenvalues or well-separated clusters of large eigenvalues can lead to acceleration of convergence,
see, e.g., \cite{LiesenS,Sluis} or the example in~\cite[Section~2]{G}.
Using finite precision arithmetic, however, can slow down the 
convergence rate; see, e.g.~\cite{MeurantStrakos,Strakos91}
and the recent comprehensive paper~\cite{GerStr}.
In any case, controlling or estimating not only condition numbers but also distributions of all eigenvalues of preconditioned matrices can yield faster convergence.
Our approach can also provide guaranteed easily accessible 
lower bounds on the 
smallest eigenvalue of the preconditioned problem, which 
is demanded, for example, for accurate algebraic error 
estimates; see, e.g., \cite{MeurantT}.

The structure of the paper is as follows. 
In the subsequent section, we introduce the diffusion and linear elasticity equations as examples of
scalar and vector valued elliptic differential equations
which our approach can be applied to.
In the third section, the discretization and the preconditioning 
setting are described.
In the fourth section, the main part of the paper, we 
suggest a method of estimating the eigenvalues 
of the preconditioned matrices.
The theoretical developments are accompanied with 
illustrative examples.
Finally, we compare our method with the recent results from~\cite{G}.
A short conclusion summarizes the paper.

\section{Diffusion and elasticity problems}\label{sec2}

Our theory of estimating the eigenvalues will be applied to two
frequent types of scalar and vector valued
elliptic partial differential equations: the  
diffusion and linear elasticity equations, respectively.
To this end, let us briefly introduce the associated 
definitions and notation; see, e.g., \cite{Blaheta1994,CiarletEla,ErnG,NecasH}
for further details. We assume general mixed
boundary conditions for the diffusion equation,
and for simplicity of exposition, homogeneous Dirichlet boundary 
conditions for the elasticity equation.

Let $\Omega\subset \mathbb{R}^d$ be a polygonal bounded domain,
where $d=2$ or $3$.
We consider the {\it diffusion equation} with Dirichlet and Robin
boundary conditions
\begin{equation*}
\nabla\cdot \bm{A}\nabla u=f\;\text{in}\;\Omega,\quad
u=g_1\; \text{on}\;\partial\Omega_1,\quad
\bm{n}\cdot\bm{A}\nabla u=g_2-g_3u\; \text{on}\;\partial\Omega_2,\quad
\end{equation*}
where $\partial\Omega_1$ and $\partial\Omega_2$ are two disjoint 
parts of the boundary $\partial\Omega$, 
${\partial\Omega}={\partial\Omega_1}\cup{\partial\Omega_2}$, and $\bm{n}$ denotes the 
outer normal to $\partial\Omega_2$.
After lifting the solution
$u$ by a~differentiable function $u_0$ that fulfills
the non-homogeneous Dirichlet boundary condition and substituting
$u:=u+u_0$, the weak 
form of the new problem reads: find 
$u\in V=\{v\in H^1(\Omega);\; v=0\;\text{on}\; \partial{\Omega}_1   \}$ such that
\begin{equation}\label{p1}
(u,v)_A=l_{A,f}(v),\quad v\in V,
\end{equation}
where
\begin{eqnarray}
(u,v)_A&=&
\int_{\Omega}\nabla v\cdot \bm{A}\nabla {u}\,{\rm d}\bm{x}
+\int_{\partial\Omega_2}g_3uv\, {\rm d}S,\nonumber\\
l_{A,f}(v)&=&
\int_{\Omega}fv\,{\rm d}\bm{x}-
\int_\Omega\nabla v\cdot \bm{A}\nabla u_0\,{\rm d}\bm{x}
+\int_{\partial\Omega_2}g_2v
\; {\rm d}S
+\int_{\partial\Omega_2}\bm{n}\cdot \bm{A}\nabla u_0 v\,{\rm d}S,
\nonumber
\end{eqnarray}
for $u,v\in V$; see, e.g., \cite{ErnG} for details. We assume $f\in L^2(\Omega)$, $g_2\in L^2(\partial\Omega_2)$, and $g_3\in L^\infty(\partial\Omega_2)$, $g_3(\bm{x})\ge 0$ on $\partial\Omega_2$.
The material data $\bm{A}:\Omega\to\mathbb{R}^{d\times d}$ 
are assumed to be essentially bounded, i.e.~$\bm{A}\in L^\infty(\Omega;\, \mathbb{R}^{d\times d})$, symmetric, and uniformly elliptic 
(positive definite) in~$\Omega$. Thus there exist constants $0<c_A\le C_A<\infty$ such that
\begin{align}\label{cACA}
c_A\Vert \bm{v} \Vert^2_{\mathbb{R}^d}
\le
(\bm{A}(\bm{x})\bm{v},\bm{v})_{\mathbb{R}^d}
\le
C_A\Vert\bm{v}\Vert^2_{\mathbb{R}^d},
\quad \bm{x} \in \Omega ,\;\bm{v}\in \mathbb{R}^d.
\end{align}

The weak form of the {\it linear elasticity problem} with 
homogeneous boundary conditions reads: find $\bm{u}\in V^d_0$,
$V_0= \{v\in H^1(\Omega);\; v=0\;\text{on}\; \partial{\Omega}\}$, such that
\begin{equation}\label{p2}
(\bm{u},\bm{v})_C
=
l_{C,F}( \bm{v}),\quad \bm{v} \in V^d_0,
\end{equation}
where
\begin{eqnarray}
(\bm{u},\bm{v})_C&=&\int_{\Omega}\sum_{i,j,k,l=1}^dc_{ijkl}\frac{\partial u_k}
{\partial x_l}\frac{\partial v_i}{\partial x_j}\,{\rm d}{\bm x},\nonumber\\ 
l_{C,F}(\bm{v})&=&-\int_{\Omega}\sum_{i=1}^d F_iv_i\,{\rm d}\bm{x},\nonumber
\end{eqnarray}
for $\bm{u}, \bm{v}\in V_0^d$, 
where $\bm{F}\in (L^2(\Omega))^d$ are body forces.
Due to the homogeneous Dirichlet boundary conditions on 
$\partial\Omega_1=\partial\Omega$,
we use the special notation $V_0$ of the solution space. Let
\begin{equation}\label{Cauchystress}
\tau_{ij}=\sum_{k,l=1}^dc_{ijkl}\, e_{kl}(\bm{u}),\quad
i,j=1,\dots,d,
\end{equation}
be the components of the Cauchy stress tensor $\bm \tau$ with
the strain components $e_{ij}$ obtained from the displacement vector $\bm{u}$ as 
\begin{equation*}
e_{kl}(\bm{u})=\frac{1}{2}\left(\frac{\partial u_k}{\partial x_l}+
\frac{\partial u_l}{\partial x_k}\right),\quad
k,l=1,\dots,d.
\end{equation*}
Assuming $d=3$ and denoting $e_i=e_{ii}$, $i=1,\dots,d$, we can write
\begin{equation}
\bm{e}=\left(
\begin{array}{c}
e_1\\
e_2\\
e_3\\
2e_{12}\\
2e_{23}\\
2e_{31}\end{array}\right)=
\left(\begin{array}{ccc}
\frac{\partial}{\partial x_1}&0&0\\
0&\frac{\partial}{\partial x_2}&0\\
0&0&\frac{\partial}{\partial x_3}\\
\frac{\partial}{\partial x_2}&\frac{\partial}{\partial x_1}&0\\
0&\frac{\partial}{\partial x_3}&\frac{\partial}{\partial x_2}\\
\frac{\partial}{\partial x_3}&0&\frac{\partial}{\partial x_1}\\
\end{array}
\right) \left(
\begin{array}{c}
u_1\\
u_2\\
u_3\end{array}\right)=\bm{\partial u}.\nonumber
\end{equation}
We assume that the coefficients $c_{ijkl}$ of the tensor $\bm{c}$ in~(\ref{Cauchystress}) 
are bounded measurable functions defined in $\Omega$,
$c_{ijkl}\in L^\infty(\Omega)$, fulfilling the symmetry conditions
\begin{equation}\label{sym_c}
c_{ijkl}=c_{jikl}=c_{klij},\quad i,j,k,l=1,\dots,d.
\end{equation}
Further, we assume there exists a constant $\mu>0$ such that
\begin{equation*}
\mu\sum_{i,j=1}^d\xi_{ij}^2\le \sum_{i,j,k,l=1}^dc_{ijkl}(\bm{x})\xi_{ij}
\xi_{kl}\quad \text{for all symmetric tensors}\; \bm{\xi} \in 
{\mathbb R}^{d\times d},\, \bm{x}\in \Omega.
\end{equation*}
Assuming $d=3$, due to the symmetries~\eqref{sym_c} of $\bm c$, 
there exist coefficients $c_{ij}\in L^\infty(\Omega)$, 
$i,j=1,\dots,6$, such that
the stress vector $\bm\tau$ can be obtained from the strain vector as 
\begin{equation}
\bm{\tau}=\left(
\begin{array}{c}
\tau_1\\
\tau_2\\
\tau_3\\
\tau_{12}\\
\tau_{23}\\
\tau_{31}\end{array}\right)=
\left(\begin{array}{cccccc}
c_{11} & c_{12} & c_{13} & c_{14} & c_{15} & c_{16} \\
c_{12} & c_{22} & c_{23} & c_{24} & c_{25} & c_{26} \\
c_{13} & c_{23} & c_{33} & c_{34} & c_{35} & c_{36} \\
c_{14} & c_{24} & c_{34} & c_{44} & c_{45} & c_{46} \\
c_{15} & c_{25} & c_{35} & c_{45} & c_{55} & c_{56} \\
c_{16} & c_{26} & c_{36} & c_{46} & c_{56} & c_{66}  
\end{array}
\right) \left(
\begin{array}{c}
e_1\\
e_2\\
e_3\\
2e_{12}\\
2e_{23}\\
2e_{31}\end{array}\right)=\bm{Ce}.\nonumber
\end{equation}
Starting from this place, we will use only the new set of material 
coefficients $c_{ij}$, $i,j=1,\dots,6$, 
(instead of $c_{ijkl}$, $i,j,k,l=1,\dots,d$)
and call the associated matrix $\bm C$.
Some special material qualities imply certain
structures of $\bm C$.
For example, homogeneous cubic 3D material corresponds to
$c_{11}=c_{22}=c_{33}$, $c_{44}=c_{55}=c_{66}$,
$c_{12}=c_{13}=c_{23}$, and annihilates the other components,
where $c_{11}>c_{12}$, $c_{11}+2c_{12}>0$ and $c_{44}>0$.
Especially, for isotropic material, we have
\begin{equation*}
c_{11}=\frac{E(1-\nu)}{(1+\nu)(1-2\nu)},\quad
c_{12}=\frac{E\nu}{(1+\nu)(1-2\nu)},\quad
c_{44}=\frac{E}{2(1+\nu)},
\end{equation*}
where $E>0$ is the Young's modulus and $\nu\in (-1,\frac{1}{2})$ is the Poisson ratio~\cite{NecasH}. 

The vector $\bm F$ of external forces fulfills  
\begin{equation*}
-\bm{\partial}^T\bm{\tau}=
-
\left(\begin{array}{cccccc}
\frac{\partial}{\partial x_1}&0&0&\frac{\partial}{\partial x_2}&0&\frac{\partial}{\partial x_3}\\
0&\frac{\partial}{\partial x_2}&0&\frac{\partial}{\partial x_1}&\frac{\partial}{\partial x_3}&0\\
0&0&\frac{\partial}{\partial x_3}&0&\frac{\partial}{\partial x_2}&\frac{\partial}{\partial x_1}
\end{array}
\right)\left(
\begin{array}{c}
\tau_1\\
\tau_2\\
\tau_3\\
\tau_{12}\\
\tau_{23}\\
\tau_{31}\end{array}\right)=\left(
\begin{array}{c}
F_1\\
F_2\\
F_3\end{array}\right)=\bm{F}
\end{equation*}
yielding 
\begin{equation*}
-\bm{\partial}^T\bm{C\partial u}=\bm{F}.
\end{equation*}
Thus $(\bm{u},\bm{v})_C$ and $l_{C,F}(\bm{v})$
can be equivalently written as
\begin{eqnarray}
(\bm{u},\bm{v})_C&=&\int_{\Omega}(\bm{\partial v})^T\bm{C\partial u}\,{\rm d}\bm{x}\nonumber\\
l_{C,F}(\bm{v})&=&\int_{\Omega}\bm{ v}^T\bm{F}\,{\rm d}\bm{x}.\nonumber
\end{eqnarray}
If $d=2$, the dimensions of the arrays naturally reduce. For example,
for cubic material we get
\begin{equation}\nonumber
\bm{u}=\left(
\begin{array}{c}
u_1\\
u_2\end{array}\right),\quad
\bm{\tau}=\left(
\begin{array}{c}
\tau_1\\
\tau_2\\
\tau_{12}\end{array}\right),\quad 
\bm{\partial}=
\left(\begin{array}{cc}
\frac{\partial}{\partial x_1}&0\\
0&\frac{\partial}{\partial x_2}\\
\frac{\partial}{\partial x_2}&\frac{\partial}{\partial x_1}
\end{array}
\right),\quad
\bm{C}=\left(\begin{array}{ccc}
c_{11} & c_{12} & 0 \\
c_{12} & c_{11} & 0 \\
0 & 0 & c_{44} 
\end{array}
\right).
\end{equation}

\section{Discretization and preconditioning}\label{Sec_Discret}

We assume that a conforming FE method is employed
to discretization of the diffusion and elasticity problems
defined by~\eqref{p1} and~\eqref{p2}, respectively. 
The domain $\Omega$ is thus decomposed into a finite number of elements 
$\mathcal{E}_j$, $j=1,\dots,N_e$. Some 
continuous FE basis functions (with compact supports) denoted by
 $\varphi_k$, $k=1,\dots,N$, are used
 as approximation and test functions.
 By $\mathcal{P}_k$ we denote
the smallest patch of elements covering the support of $\varphi_k$.  
 Correspondingly to section~\ref{sec2}, 
 we denote the material data by $\bm{A}$ and $\bm{C}$ 
of the diffusion and elasticity operators, respectively,
and the data of the associated preconditioning operators 
by $\widetilde{\bm{A}}$ and $\widetilde{\bm{C}}$, respectively.
The function $g_3$ entering the Robin boundary conditions is 
allowed to be 
different in the original and preconditioning operators;
therefore, it is denoted by $\widetilde{g}_3$ in the latter. 

The stiffness matrices $\mathsf{A}$ and $\mathsf{C}$
of the systems of linear equations of the discretized
problems~\eqref{p1} and~\eqref{p2}, respectively,
have elements
\begin{equation*}
\mathsf{A}_{kl}=\int_{\Omega}\nabla \varphi_{l}(\bm{x})\cdot\bm{A}(\bm{x})\nabla \varphi_{k}({\bm{x}})\,{\text d}\bm{x}+
\int_{\partial\Omega_2}g_3(\bm{x})\varphi_{l}(\bm{x})\varphi_{k}({\bm{x}})\,{\text d}S
\end{equation*}
and
\begin{equation}\label{precC}
\mathsf{C}_{\bm{kl}}=\int_{\Omega}
(\bm{\partial} (\varphi_{l_1}(\bm{x}),\dots,\varphi_{l_d}(\bm{x}))^T)^T
\bm{C}(\bm{x})\bm{\partial} (\varphi_{k_1}(\bm{x}),\dots,\varphi_{k_d}(\bm{x}))^T\,{\text d}\bm{x},
\end{equation}
respectively, where $k,l=1,\dots,N$, and $\bm{k},\bm{l}\in \{1,\dots,N\}^d$.
The preconditioning matrices
$\widetilde{\mathsf{A}}$ and $\widetilde{\mathsf{C}}$
obtained for the material data $\widetilde{\bm{A}}$ and
$\widetilde{\bm{C}}$, respectively, have elements
\begin{equation*}
\widetilde{\mathsf{A}}_{kl}=\int_{\Omega}\nabla \varphi_{l}(\bm{x})\cdot\widetilde{\bm{A}}(\bm{x})\nabla \varphi_{k}({\bm{x}})\,{\text d}\bm{x}+\int_{\partial\Omega_2}{\widetilde{g}_3}(\bm{x})\varphi_{l}(\bm{x})\varphi_{k}({\bm{x}})\,{\text d}S
\end{equation*}
and
\begin{equation*}
\widetilde{\mathsf{C}}_{\bm{kl}}=\int_{\Omega}
(\bm{\partial} (\varphi_{l_1}(\bm{x}),\dots,\varphi_{l_d}(\bm{x}))^T)^T
\widetilde{\bm{C}}(\bm{x})\bm{\partial} (\varphi_{k_1}(\bm{x}),\dots,\varphi_{k_d}(\bm{x}))^T\,{\text d}\bm{x},
\end{equation*}
respectively. 
All integrals are supposed to be carried out exactly.

The idea of preconditioning, see, e.g.~\cite[Section~10.3]{Golub}
or~\cite[Chapters~9 and~10]{SaadBook}, is based on 
assumptions that 
a system of linear equations with a matrix $\widetilde{\mathsf{M}}$ is relatively easily solvable and that the spectrum of $\widetilde{\mathsf{M}}^{-1}\mathsf{M}$ is more favorable than
that of $\mathsf{M}$ regarding some iterative solution method,
which does not necessarily mean a smaller condition number~\cite{G}.
 Substituting the equation
$\mathsf{Mu}=\mathsf{B}$ with 
\begin{equation*}
\widetilde{\mathsf{M}}^{-1}\mathsf{Mu}=\widetilde{\mathsf{M}}^{-1}\mathsf{B}\quad \text{or}\quad 
\widetilde{\mathsf{M}}^{-1/2}\mathsf{M}\widetilde{\mathsf{M}}^{-1/2}\mathsf{v}=\widetilde{\mathsf{M}}^{-1/2}\mathsf{B},\;
\mathsf{u}=\widetilde{\mathsf{M}}^{-1/2}\mathsf{v},
\end{equation*}
thus leads to equivalent problems that can be solved more
efficiently than the original one.

\section{Bounds on eigenvalues of preconditioned problems}

The main results of the paper are introduced in this section.
Instead of presenting our results for a general
elliptic second order partial differential equation with 
tensor data and a vector valued unknown function $\bm{u}$, 
we first present our theory for the
(scalar) diffusion equation with tensor data in full detail.
Then we apply the same approach to the elasticity equation.
The section is concluded by some general remarks mainly on relationship
between our results and the recent results from~\cite{G}.

\subsection{Diffusion equation}\label{sec_diff}

The lower and upper bounds on the eigenvalues 
$0\le\lambda_1\le\dots\le\lambda_N$ of 
$\widetilde{\mathsf{A}}^{-1}\mathsf{A}$
for any uniformly positive definite measurable
data $\bm{A},\widetilde{\bm{A}}:\Omega\to \mathbb{R}^{d\times d}$
are introduced in this part.
The boundary conditions of the original and preconditioning problems
may differ at most in the function $g_3$, i.e.~instead of $g_3$, the function 
$\widetilde{g}_3$ can be used in Robin boundary condition of
the preconditioning problem. We assume, however, that there exist
constants $0<c_g\le C_g<\infty$ such that  
\begin{equation*}
0\le c_g\,\widetilde{g}_3(\bm{x})\le {g}_3(\bm{x})\le C_g\,\widetilde{g}_3(\bm{x}),\quad \bm{x}\in \partial\Omega_2.
\end{equation*}
Since $N$ is the number of the
FE basis functions then $\mathsf{A},\widetilde{\mathsf{A}}\in\mathbb{R}^{N\times N}$.
We now build two sequences of positive real numbers
$\lambda^{\rm L}_k$  and $\lambda^{\rm U}_k$, 
$k=1,\dots,N$. Let us first set
\begin{eqnarray}
\alpha^{\rm min}_j&=&{\rm ess\, inf}_{\bm{x}\in \mathcal{E}_j}\;\lambda_{\rm min}
\left(\widetilde{\bm{A}}^{-1}(\bm{x}){\bm{A}}(\bm{x})\right),\nonumber\\
\alpha^{\rm max}_j&=&{\rm ess\, sup}_{\bm{x}\in \mathcal{E}_j}\;\lambda_{\rm max}
\left(\widetilde{\bm{A}}^{-1}(\bm{x}){\bm{A}}(\bm{x})\right),
\nonumber
\end{eqnarray}
if no edge of $\mathcal{E}_j$ lies in $\partial\Omega_2$, and
\begin{eqnarray}
&&\alpha^{\rm min}_j\nonumber\\
&&=\min\left\{
{\rm ess\, inf}_{\bm{x}\in \partial\Omega_2\cap\overline{\mathcal{E}}_j,\, g_3(\bm{x})\ne 0}\;\widetilde{g}_3^{-1}(\bm{x})g_3(\bm{x}),\;
{\rm ess\, inf}_{\bm{x}\in \mathcal{E}_j}\;\lambda_{\rm min}
\left(\widetilde{\bm{A}}^{-1}(\bm{x}){\bm{A}}(\bm{x})\right)\right\},\nonumber\\
&&\alpha^{\rm max}_j\nonumber\\
&&=\max\left\{{\rm ess\, sup}_{\bm{x}\in \partial\Omega_2\cap\overline{\mathcal{E}}_j,\, g_3(\bm{x})\ne 0}\;\widetilde{g}_3^{-1}(\bm{x})g_3(\bm{x}),\;{\rm ess\, sup}_{\bm{x}\in \mathcal{E}_j}\;\lambda_{\rm max}
\left(\widetilde{\bm{A}}^{-1}(\bm{x}){\bm{A}}(\bm{x})\right)\right\}
\nonumber
\end{eqnarray}
if at least one edge of $\mathcal{E}_j$ lies in $\partial\Omega_2$, $j=1,\dots,N_e$.
If $\bm{A}(\bm{x})$ and $\widetilde{\bm{A}}(\bm{x})$
are element-wise constant and if $g_3$
and $\widetilde{g}_3$ are constant on every edge (of any element) 
lying in 
$\partial\Omega_2$, the computation of $\alpha^{\min}_j$
and $\alpha^{\max}_j$
reduces to calculating the extreme eigenvalues of $d\times d$
matrices on all individual elements $\mathcal{E}_j$, $j=1,\dots,N_e$,
and eventual comparing them with $\widetilde{g}_3^{-1}(\bm{x})g_3(\bm{x})$
on some of the attached edges.
For every function $\varphi_k$, supported on the patch 
$\mathcal{P}_k$, let us set 
\begin{equation}\label{lam12}
\lambda^{\rm L}_k=\min_{\mathcal{E}_j\subset \mathcal{P}_k}
\alpha^{\rm min}_j,\quad
\lambda^{\rm U}_k=\max_{\mathcal{E}_j\subset \mathcal{P}_k}
\alpha^{\rm max}_j, \quad j=1,\dots,N.
\end{equation}
Thus $\lambda^{\rm L}_k$ and $\lambda^{\rm U}_k$ are in the above sense 
the smallest and the largest, respectively, eigenvalues of  
$\widetilde{\bm{A}}^{-1}(\bm{x}){\bm{A}}(\bm{x})$ on the patch $\mathcal{P}_k$, or the extremes of $\widetilde{g}_3^{-1}g_3$
along the parts of the boundary of $\mathcal{P}_k$ lying in 
$\partial\Omega_2$.
After inspecting all patches, we sort the two series in~\eqref{lam12}
non-decreasingly. Thus we obtain two bijections 
\begin{equation*}r,s:\{1,\dots,N\}\to \{1,\dots,N\}
\end{equation*} such that
\begin{equation}\label{LUbounds}
\lambda_{r(1)}^{\rm L}\le \lambda_{r(2)}^{\rm L}\le \dots\le 
\lambda_{r(N)}^{\rm L},\qquad
\lambda_{s(1)}^{\rm U}\le \lambda_{s(2)}^{\rm U}\le \dots\le 
\lambda_{s(N)}^{\rm U}.
\end{equation}
Note that we could define and compute
$\lambda_k^{\rm L}$ and $\lambda_k^{\rm U}$
directly without defining 
$\alpha^{\rm min}_j$ and  $\alpha^{\rm max}_j$. 
However, dealing with the constants 
$\alpha^{\rm min}_j$ and  $\alpha^{\rm max}_j$ is more 
algorithmically acceptable, because it allows to avoid multiple evaluation of eigenvalues of $\widetilde{\bm{A}}^{-1}\bm{A}$ 
on every element.

Next we prove an auxiliary lemma.
Let $\sigma(\mathsf{M})$ denote the spectrum of the matrix $\mathsf{M}$.
\begin{lem}\label{lem1}
Let $\bm{A}(\bm{x}),\widetilde{\bm{A}}(\bm{x})\in \mathbb{R}^{d\times d}$ be symmetric and positive definite for all $\bm{x}\in \mathcal{D}\subset\Omega$. 
Let there exist constants $0<c_1\le c_2<\infty$ and $0<c_3\le c_4<\infty$
such that
\begin{equation}\label{lem1c}
\sigma(\widetilde{\bm{A}}^{-1}(\bm{x})\bm{A}(\bm{x}))\subset [ c_1,c_2],
\quad \bm{x}\in \mathcal{D},
\end{equation}
and 
\begin{equation*}0\le c_3\,\widetilde{g}_3(\bm{x})\le g_3(\bm{x})\le 
c_4\,\widetilde{g}_3(\bm{x}), \quad \bm{x}\in \partial\Omega_2
\cap\overline{\mathcal{D}}.
\end{equation*}
Then for $u\in H_0^1(\Omega)$ we get
\begin{equation}\label{lem1a}
c_1\int_\mathcal{D} \nabla u\cdot\widetilde{\bm{A}}\nabla u\,{\rm d}\bm{x}\le
\int_\mathcal{D} \nabla u\cdot\bm{A}\nabla u\,{\rm d}\bm{x}\le c_2
\int_\mathcal{D} \nabla u\cdot\widetilde{\bm{A}}\nabla u\,{\rm d}\bm{x}
\end{equation}
and
\begin{eqnarray}
&&\min\{c_1,c_3\}\left(\int_\mathcal{D} \nabla u\cdot\widetilde{\bm{A}}\nabla u\,{\rm d}\bm{x}+
\int_{\partial\Omega_2\cap\overline{\mathcal{D}}}\widetilde{g}_3u^2\,{\rm d}S
\right)\nonumber\\
&&\le\int_\mathcal{D} \nabla u\cdot\bm{A}\nabla u\,{\rm d}\bm{x}+
\int_{\partial\Omega_2\cap\overline{\mathcal{D}}}g_3u^2\,{\rm d}S\label{lem1b}\\
&&\le 
\max\{c_2,c_4\}
\left(\int_\mathcal{D} \nabla u\cdot\widetilde{\bm{A}}\nabla u\,{\rm d}\bm{x}+
\int_{\partial\Omega_2\cap\overline{\mathcal{D}}}\widetilde{g}_3u^2\,{\rm d}S
\right).\nonumber
\end{eqnarray}
\end{lem}
\begin{proof}
Since for all $\bm{v}\in \mathbb{R}^d$ and $\bm{x}\in \mathcal{D}$
it follows from~\eqref{lem1c} that
\begin{equation*}
c_1\,\bm{v}^T\widetilde{\bm{A}}(\bm{x})\bm{v}\le \bm{v}^T\bm{A}(\bm{x})
\bm{v}\le c_2\,\bm{v}^T\widetilde{\bm{A}}(\bm{x})\bm{v},
\end{equation*}
we get~\eqref{lem1a} by setting $\bm{v}=\nabla u$ and 
integrating all three terms over $\mathcal D$. Inequalities~\eqref{lem1b} follow obviously using $g_3\ge 0$.
\end{proof}

Now we introduce the first part of the main results of this paper.

\begin{The}\label{thm1}
Let us assume that the $(d-1)$-dimensional measure of $\partial \Omega_1$ is positive.
The lower and upper bounds on the  
eigenvalues $0<\lambda_1\le \lambda_2\le \dots\le \lambda_N$
of $\widetilde{\mathsf{A}}^{-1}\mathsf{A}$ 
are given by~\eqref{LUbounds}, i.e.,
\begin{equation}\label{thm1a}
\lambda^{\rm L}_{r(k)}\le \lambda_k\le \lambda^{\rm U}_{s(k)},\qquad
k=1,\dots,N.
\end{equation}
\end{The}
\begin{proof}
Due to the positive measure of $\partial\Omega_1$, the matrices
$\widetilde{\mathsf{A}}$ and $\mathsf{A}$ are positive definite.
We only prove the lower bounds of~\eqref{thm1a}; 
the upper bounds can be
proved analogously.
Due to the Courant--Fischer min-max theorem, e.g.~\cite[Theorem~8.1.2]{Golub}, 
\begin{equation*}
\lambda_k=\max_{S,\, {\rm dim}S=N-k+1}\, \min_{\mathsf{v}\in S,\, \mathsf{v}
\ne \mathsf{0} }
\frac{\mathsf{v}^T\mathsf{Av}}{\mathsf{v}^T\widetilde{\mathsf{A}}\mathsf{v}},
\end{equation*}
where $S$ denotes a subspace of $\mathbb{R}^N$.
Then we have
\begin{eqnarray}
\lambda_1&=& \max_{S,\, {\rm dim}S=N}\, \min_{\mathsf{v}\in S,\, \mathsf{v}\ne \mathsf{0} }
\frac{\mathsf{v}^T\mathsf{Av}}{\mathsf{v}^T\widetilde{\mathsf{A}}\mathsf{v}}=
 \min_{\mathsf{v}\in \mathbb{R}^N,\, \mathsf{v}\ne \mathsf{0} }
\frac{\mathsf{v}^T\mathsf{Av}}{\mathsf{v}^T\widetilde{\mathsf{A}}\mathsf{v}}
\ge \lambda^{\rm L}_{r(1)}, \nonumber
\end{eqnarray}
where the inequality follows from Lemma~\ref{lem1}.
Indeed, using ${u}=\sum_{i=1}^N\mathsf{v}_i\varphi_i$,
definitions~\eqref{lam12} and Lemma~\ref{lem1} with
$\mathcal{D}=\Omega$, we get
\begin{equation*}
\frac{\mathsf{v}^T\mathsf{Av}}{\mathsf{v}^T\widetilde{\mathsf{A}}\mathsf{v}}=
\frac{\int_\Omega \nabla u\cdot{\bm{A}}\nabla u\,{\rm d}\bm{x}
+
\int_{\partial\Omega_2}{g_3}u^2\,{\text d}S
}{
\int_\Omega \nabla u\cdot\widetilde{\bm{A}}\nabla u\,{\rm d}\bm{x}+\int_{\partial\Omega_2}{\widetilde{g}_3}u^2\,{\text d}S}
\ge \min_{\mathcal{E}_j\subset\Omega}\alpha^{\rm min}_j
=\min_{\mathcal{P}_k\subset\Omega}\lambda^{\rm L}_k
=\lambda^{\rm L}_{r(1)}.
\end{equation*}
Then we proceed to
\begin{eqnarray}
\lambda_2&=& \max_{S,\, {\rm dim}S=N-1}\, \min_{\mathsf{v}\in S,\, \mathsf{v}\ne \mathsf{0} }
\frac{\mathsf{v}^T\mathsf{Av}}{\mathsf{v}^T\widetilde{\mathsf{A}}\mathsf{v}}\ge
 \min_{\mathsf{v}\in \mathbb{R}^N,\, \mathsf{v}\ne \mathsf{0},\, \mathsf{v}_{r(1)} =0}
\frac{\mathsf{v}^T\mathsf{Av}}{\mathsf{v}^T\widetilde{\mathsf{A}}\mathsf{v}}\ge \lambda^{\rm L}_{r(2)}, \nonumber
\end{eqnarray}
where the last inequality follows from Lemma~\ref{lem1} where 
(due to  $\mathsf{v}_{r(1)}=0$)
$\mathcal{D}$ contains only the patches associated to the FE
basis functions $\varphi_j$, $j\ne r(1)$,
\begin{equation*}
\mathcal{D}=\cup_{j\in\{1,\dots,N\}\setminus\{ r(1)\}}\mathcal{P}_j,
\end{equation*}
and from
\begin{eqnarray}
\min_{\mathsf{v}\in \mathbb{R}^N,\, \mathsf{v}\ne \mathsf{0},\, \mathsf{v}_{r(1)} =0}
\frac{\mathsf{v}^T\mathsf{Av}}{\mathsf{v}^T\widetilde{\mathsf{A}}\mathsf{v}}&&=
\min_{u=\sum_{i=1}^N\mathsf{v}_i\varphi_i,\, \mathsf{v}_{r(1)}=0}
\frac{\int_\mathcal{D} \nabla u\cdot{\bm{A}}\nabla u\,{\rm d}\bm{x}
+
\int_{\partial\Omega_2\cap\overline{\mathcal{D}}}{g_3}u^2\,{\text d}S
}{
\int_\mathcal{D} \nabla u\cdot\widetilde{\bm{A}}\nabla u\,{\rm d}\bm{x}+\int_{\partial\Omega_2\cap\overline{\mathcal{D}}}{\widetilde{g}_3}u^2\,{\text d}S}
\nonumber\\
&&\ge\min_{\mathcal{E}_j\subset\mathcal{D}}\alpha^{\rm min}_j
=\min_{\mathcal{P}_k\subset\mathcal{D}}\lambda^{\rm L}_k
=\lambda^{\rm L}_{r(2)}.\nonumber
\end{eqnarray}
We can proceed further in the same manner to get all inequalities
$\lambda^{\rm L}_{r(k)}\le \lambda_k$ of~\eqref{thm1a}. 
\end{proof}

In Theorem~\ref{thm1}, we consider positive definite
problems with homogeneous Dirichlet and/or general
Robin boundary conditions (with $g_3\ge 0$). 
Neumann boundary condition is a special
type of Robin boundary condition with $g_3=0$.
In practical implementation of nonhomogeneous Dirichlet boundary 
conditions, the lifting function $u_0$ does not necessarily
have to be employed.
If the same non-homogeneous Dirichlet boundary conditions
are considered for the original and preconditioning problems,
the method of getting 
the lower and upper bounds~\eqref{LUbounds} can be used unchanged.
Our theory, however, does not cover the settings where the original and 
preconditioning problems are considered under 
different non-homogeneous Dirichlet
boundary conditions or different functions $g_2$ in Robin boundary 
conditions, or if $\partial \Omega_1$ in the preconditioning problem does not
coincide with $\partial \Omega_1$ used for the original problem.

If periodic or Neumann boundary conditions are applied 
along $\partial \Omega$ and if they are the same for 
the original and preconditioning problems,
then $\mathsf A$ and $\widetilde{\mathsf{A}}$ are singular;
they share the smallest eigenvalue $\lambda_1=0$ 
and the associated eigenvector.
Then we can use the same method again to get the bounds on all of the eigenvalues of the preconditioned matrix;
however, we must omit the null space of $\mathsf{A}$ 
(which is the same as the null space of $\widetilde{\mathsf{A}}$) from the respective formulas. 
To justify the method, we can proceed analogously
as in the proof of Theorem~\ref{thm1}, where the vectors $\mathsf v$ are
now additionally considered
fulfilling $\widetilde{\mathsf{A}}\mathsf{v}\ne 0$. Then 
\begin{equation*}
\lambda_2\ge \min_{\mathsf{v}\in \mathbb{R}^N,\, \widetilde{\mathsf{A}}\mathsf{v}
\ne \mathsf{0} }
\frac{\mathsf{v}^T\mathsf{Av}}{\mathsf{v}^T\widetilde{\mathsf{A}}\mathsf{v}}\ge \lambda^{\rm L}_{r(1)}.
\end{equation*}
We can proceed further, analogously to the proof of Theorem~\ref{thm1},
\begin{equation*}
\lambda_3\ge
 \min_{\mathsf{v}\in \mathbb{R}^N,\,  \widetilde{\mathsf{A}}\mathsf{v}\ne \mathsf{0},\, \mathsf{v}_{r(1)} =0}
\frac{\mathsf{v}^T\mathsf{Av}}{\mathsf{v}^T\widetilde{\mathsf{A}}\mathsf{v}}
\ge \lambda^{\rm L}_{r(2)}. \nonumber
\end{equation*}
In this way we get $N-1$ lower bounding numbers on the 
non-zero eigenvalues of $\widetilde{\mathsf{A}}^{-1}\mathsf{A}$,
where both $\mathsf{A}$ and $\widetilde{\mathsf{A}}$ are now
considered restricted to the subspace of $\mathbb{R}^N$ 
that is orthogonal to the null space of $\mathsf{A}$.
Analogously, we get the upper bounds; thus finally,
\begin{equation*}
\lambda^{\rm L}_{r(k-1)}\le \lambda_k\le \lambda^{\rm U}_{s(k)},\quad k=2,\dots,N.
\end{equation*}

Let us now apply our method to some examples.
\begin{ex} \label{ex111}
Assume $d=2$, $\Omega=(-\pi,\pi)^2$,
$\partial\Omega_2=\{\bm{x};\, x_1=\pi\}$, 
$$\bm{A}(\bm{x})=\left(\begin{array}{cc}
1+0.3\,{\rm sign}(\sin(x_2))   &  0.3+0.1\cos(x_1)\\
0.3+0.1\cos(x_1)   &     1+0.3\,{\rm sign}(\sin(x_2))
\end{array}\right),$$
and a simple and a more sophisticated preconditioning operators with
$$\widetilde{\bm{A}}_1(\bm{x})=\left(\begin{array}{cc}
1&0\\
0& 1\end{array}\right),\quad\text{and}\quad 
\widetilde{\bm{A}}_2(\bm{x})=\left(\begin{array}{cc}
1&0.3\\
0.3& 1\end{array}\right),
$$
respectively. Let us consider one of the following settings:\\
(a) uniform grid with piece-wise bilinear FE functions, $N=10^2$ or
$30^2$, $g_3=0$; see Figure~\ref{fff1}; \\
(b) uniform grid with piece-wise bilinear FE functions, 
periodic boundary conditions, $N=21^2$; see Figure~\ref{fff2}; \\
(c) nonuniform grid and triangular elements with piece-wise linear
FE functions, $g_3=\widetilde{g}_3=1+x_2^2$, $N=400$; see Figure~\ref{fff3}. \\
The numerical experiments illustrate that the bounds on the 
eigenvalues are guaranteed for different types of boundary conditions. We can also notice that since ${\bm{A}}$
is point-wise closer to $\widetilde{\bm{A}}_2$ than to $\widetilde{\bm{A}}_1$,
the spectrum of the second preconditioned problem 
(together with its bounds) is closer to unity than the spectrum of the problem preconditioned by using $\widetilde{\bm{A}}_1$.
Note also that refining the mesh does not lead to more accurate 
bounds, in general. This is caused by the difference between the extreme eigenvalues
of $\widetilde{\bm{A}}_i^{-1}\bm{A}$, $i=1,2$, on individual elements;
see also section~\ref{Sec33}. 

The numbers of the CG steps needed to reduce the energy norm of the 
errors by the factor $10^{-9}$ (starting with zero
initial vectors) for setting (a) with $f=1$ in $\Omega$ 
are 17 and  13 for $\widetilde{\bm{A}}_1$ and $\widetilde{\bm{A}}_2$,
respectively, for $N=10^2$, and   20 and  15 for $\widetilde{\bm{A}}_1$ and $\widetilde{\bm{A}}_2$,
respectively, for $N=30^2$.

\begin{figure}[ht]
\vskip2pt plus 2pt \vbox to140pt{\vfil\hbox{\includegraphics{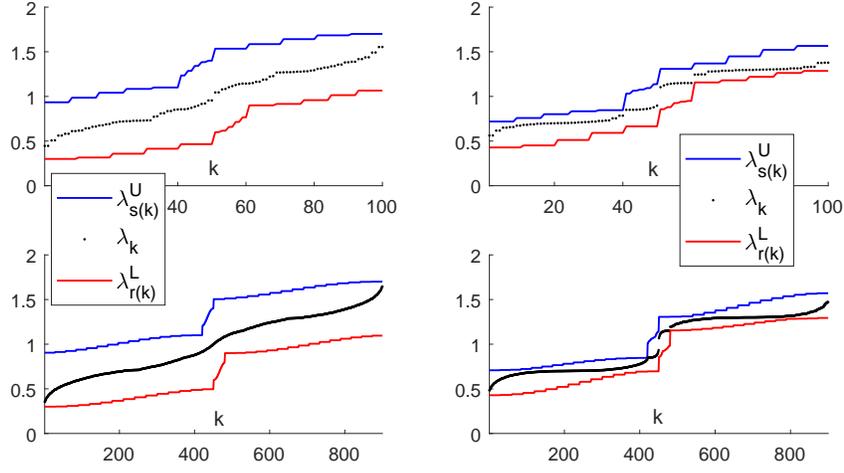}}}
\vspace{0.7in}
\caption{Lower ($\lambda^{\rm L}_{r(k)}$)
	and upper ($\lambda^{\rm U}_{s(k)}$) bounds on eigenvalues 
	$\lambda_k$ of 
	 Example~\ref{ex111}~(a)
	with $N=10^2$ (top graphs) and $N=30^2$ (bottom graphs)
	preconditioned by operators with 
	$\widetilde{\bm{A}}_1$ (left)
	and $\widetilde{\bm{A}}_2$ (right).}
\label{fff1} 
\end{figure}
\begin{figure}[ht]
\vskip2pt plus 2pt \vbox to50pt{\vfil\hbox{\includegraphics{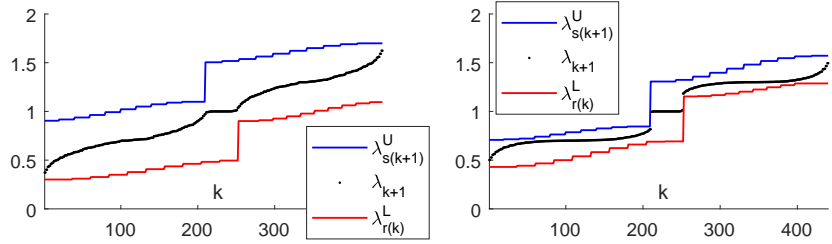}}}
\vspace{0.7in}
\caption{Lower ($\lambda^{\rm L}_{r(k)}$)
	and upper ($\lambda^{\rm U}_{s(k)}$) bounds on eigenvalues 
	$\lambda_{k}$ of 
	 Example~\ref{ex111}~(b)
	with $N=21^2$ preconditioned by operators with 
	$\widetilde{\bm{A}}_1$ (left)
	and $\widetilde{\bm{A}}_2$ (right).}
\label{fff2} 
\end{figure}
\begin{figure}[ht]
\vskip2pt plus 2pt \vbox to50pt{\vfil\hbox{\includegraphics{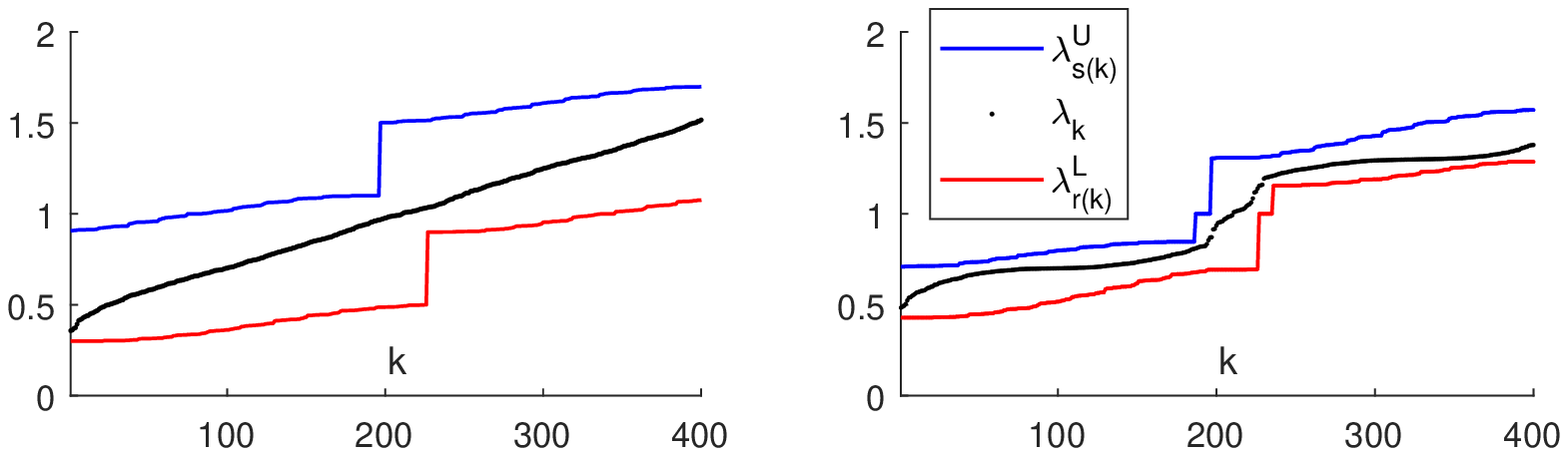}}}
\vspace{0.7in}
\caption{Lower ($\lambda^{\rm L}_{r(k)}$)
	and upper ($\lambda^{\rm U}_{s(k)}$) bounds on eigenvalues 
	$\lambda_{k}$ of 
	Example~\ref{ex111}~(c) with $N=400$ preconditioned by operators with 
	$\widetilde{\bm{A}}_1$ (left)
	and $\widetilde{\bm{A}}_2$ (right) with $g_3=\widetilde{g}_3=1+x_2^2$.}
\label{fff3} 
\end{figure}
\end{ex}

\subsection{Elasticity equation}

In the elasticity problem, or in vector valued problems in general,
the searched function has multiple components, 
$\bm{u}(\bm{x})=(u_1(\bm{x}),\dots,u_d(\bm{x}))^T$, where
individual components are coupled within the equation.
For approximation of the scalar functions $u_j$, $j=1,\dots,d$, we use 
the same sets of the FE basis functions $\varphi_k$, $k=1,\dots,N$, 
supported again inside the patches $\mathcal{P}_k$.
Recall that for the sake of simplicity, we consider 
homogeneous Dirichlet boundary conditions only.

\begin{lem}\label{lem2}
Let $\bm{C}(\bm{x}),\widetilde{\bm{C}}(\bm{x})\in \mathbb{R}^{m\times m}$, where $m=3$ if $d=2$, and $m=6$ if $d=3$.
Let $\bm C$ and $\widetilde{\bm{C}}$ 
be symmetric and positive definite for all $\bm{x}\in \mathcal{D}\subset\Omega$. 
Let there exist constants $0<c_1\le c_2<\infty$ such that 
\begin{equation}\label{C1C}
\sigma(\widetilde{\bm{C}}^{-1}(\bm{x})\bm{C}(\bm{x}))\subset [ c_1,c_2],
\quad \bm{x}\in \mathcal{D}.
\end{equation}
Then for $\bm{u}\in V_0^d$ we get
\begin{equation}\label{lem2a}
c_1\int_\mathcal{D} (\bm{\partial u})^T\cdot\widetilde{\bm{C}}\bm{\partial u} \,{\rm d}\bm{x}\le
\int_\mathcal{D} (\bm{\partial u})^T\cdot{\bm{C\partial u}} \,{\rm d}\bm{x}\le 
c_2
\int_\mathcal{D} (\bm{\partial u})^T\cdot\widetilde{\bm{C}}\bm{\partial u} \,{\rm d}\bm{x}
\end{equation}
\end{lem}
\begin{proof}
From~\eqref{C1C} for all $\bm{v}\in \mathbb{R}^d$, $\bm{x}\in \mathcal{D}$, we get
\begin{equation*}
c_1\,\bm{v}^T\widetilde{\bm{C}}(\bm{x})\bm{v}\le \bm{v}^T\bm{C}(\bm{x})
\bm{v}\le c_2\,\bm{v}^T\widetilde{\bm{C}}(\bm{x})\bm{v}.
\end{equation*}
Then by setting $\bm{v}=\bm{\partial}\bm{u}$ and integrating over $\mathcal{D}$,
we obtain~\eqref{lem2a}.
\end{proof}

We now show how to obtain the guaranteed bounds on all individual eigenvalues $0<\lambda_1\le \dots\le \lambda_{dN}$
of the preconditioned elasticity problem 
$\widetilde{\mathsf{C}}^{-1}\mathsf{C}$ for any positive definite 
material data $\bm{C}$ and $\widetilde{\bm{C}}$.
Since $N$ is the number of the FE basis functions defined on $\Omega$
used to approximate each component of $\bm u$, the number of 
unknowns is $dN$.
We now build two sequences $\lambda^{\rm L}_k$ and $\lambda^{\rm U}_k$, $k=1,\dots,dN$, to bound the eigenvalues of $\widetilde{\mathsf{C}}^{-1}\mathsf{C}$.
In contrast to subsection~\ref{sec_diff},
for the sake of brevity, we 
do not define $\alpha^{\rm min}_j$ and $\alpha^{\rm max}_j$,
but we directly set
\begin{eqnarray}
\widehat{\lambda}^{\rm L}_k&=&{\rm ess\, inf}_{\bm{x}\in \mathcal{P}_k}\;\lambda_{\rm min}
\left(\widetilde{\bm{C}}^{-1}(\bm{x}){\bm{C}}(\bm{x})
\right),\nonumber\\
\widehat{\lambda}^{\rm U}_k&=&{\rm ess\, sup}_{\bm{x}\in \mathcal{P}_k}\;
\lambda_{\rm max}
\left(\widetilde{\bm{C}}^{-1}(\bm{x}){\bm{C}}(\bm{x})\right),\nonumber
\end{eqnarray}
$k=1,\dots,N$.
Similarly to the case of the
diffusion equation in section~\ref{sec_diff}, 
we sort these two series non-decreasingly, and thus get bijections 
\begin{equation*}
R,S:\{1,\dots,N\}\to \{1,\dots,N\},
\end{equation*}
such that
\begin{equation*}
\widehat{\lambda}_{R(1)}^{\rm L}\le\dots\le\widehat{\lambda}^{\rm L}_{R(N)},
\quad
\widehat{\lambda}_{S(1)}^{\rm U}
\le\dots\le\widehat{\lambda}^{\rm U}_{S(N)}.\label{LUboundsC2}
\end{equation*}
Moreover, we double (if $d=2$) or triple (if $d=3$) all items in 
the two series of $\widehat{\lambda}^{\rm L}_k$ and $\widehat{\lambda}^{\rm U}_k$ and get two new $d$-times longer series
\begin{eqnarray}
\lambda_{(k-1)d+1}^{\rm L}=\dots=\lambda_{kd}^{\rm L}=\widehat{\lambda}^{\rm L}_k,\quad
\lambda_{(k-1)d+1}^{\rm U}=\dots=\lambda_{kd}^{\rm U}=\widehat{\lambda}^{\rm U}_k,\quad k=1,\dots,N,
\nonumber
\end{eqnarray}
that can be sorted non-decreasingly. Thus we obtain 
two bijections 
\begin{equation*}
r,s:\{1,\dots,dN\}\to \{1,\dots,dN\},
\end{equation*}
such that 
\begin{eqnarray}
\lambda_{r(1)}^{\rm L}=\dots=\lambda_{r(d)}^{\rm L}\le \lambda_{r(d+1)}^{\rm L}=
\dots =
\lambda_{r(2d)}^{\rm L}\le \dots\nonumber\\
\dots\le \lambda_{r(dN-d+1)}^{\rm L}=\dots=
\lambda_{r(dN)}^{\rm L},\label{LUboundsC1}\\
\lambda_{s(1)}^{\rm U}=\dots=\lambda_{s(d)}^{\rm U}\le \lambda_{s(d+1)}^{\rm U}=
\dots =
\lambda_{s(2d)}^{\rm U}\le \dots\nonumber\\
\dots\le \lambda_{s(dN-d+1)}^{\rm U}=\dots=
\lambda_{s(dN)}^{\rm U}.\label{LUboundsC2}
\end{eqnarray}
Note that for $k=1,\dots,N$, 
\begin{equation*}
\widehat{\lambda}^{\rm L}_{R(k)}=\lambda^{\rm L}_{r((k-1)d+1)}=
\dots =\lambda^{\rm L}_{r(kd)},\quad
\widehat{\lambda}^{\rm U}_{S(k)}=\lambda^{\rm U}_{s((k-1)d+1)}=
\dots =\lambda^{\rm U}_{s(kd)}.
\end{equation*}

Now we can introduce the second part of the main results of this paper.
\begin{The}\label{thm36}
The lower and upper bounds on all 
eigenvalues $0<\lambda_1\le \lambda_2\le \dots\le \lambda_{dN}$
of $\widetilde{\mathsf{C}}^{-1}\mathsf{C}$ 
can be obtained from~\eqref{LUboundsC1} and~\eqref{LUboundsC2}, namely
\begin{equation}\label{thm2a}
\lambda^{\rm L}_{r(k)}\le \lambda_k\le \lambda^{\rm U}_{s(k)},\qquad
k=1,\dots,dN.
\end{equation}
\end{The}
\begin{proof}
The proof is similar to the proof of Theorem~\ref{thm1}.
By the Courant--Fischer min-max theorem, 
\begin{equation*}
\lambda_k=\max_{S,\, {\rm dim}S=dN-k+1}\, \min_{\mathsf{v}\in S,\, \mathsf{v}
\ne \mathsf{0} }
\frac{\mathsf{v}^T\mathsf{Cv}}{\mathsf{v}^T\widetilde{\mathsf{C}}\mathsf{v}}.
\end{equation*}
Then
\begin{equation*}
\lambda_d\ge\dots\ge\lambda_1= \min_{\mathsf{v}\in \mathbb{R}^{dN},\, \mathsf{v}
\ne \mathsf{0} }
\frac{\mathsf{v}^T\mathsf{Cv}}{\mathsf{v}^T\widetilde{\mathsf{C}}\mathsf{v}}\ge \lambda^{\rm L}_{r(1)}=\dots=\lambda^{\rm L}_{r(d)},
\end{equation*}
where the last inequality follows from Lemma~\ref{lem2}. 
Indeed, representing the coefficients of the components of
$\bm{u}=(u_1,\dots,u_d)$ 
with respect to the FE basis functions in a single vector 
$\mathsf v=(\mathsf{v}_{(1)}^T, \dots,\mathsf{v}_{(d)}^T)^T$,
$\mathsf{v}_{(j)}\in\mathbb{R}^N$, $j=1,\dots,d$, we get
\begin{equation*}
\frac{\mathsf{v}^T\mathsf{Cv}}{\mathsf{v}^T\widetilde{\mathsf{C}}\mathsf{v}}= 
\frac{\int_\Omega (\bm{\partial u})^T\cdot{\bm{C}}\bm{\partial u} \,{\rm d}\bm{x}}
{\int_\Omega (\bm{\partial u})^T\cdot\widetilde{\bm{C}}\bm{\partial u} \,{\rm d}\bm{x}}
\ge \min_{\mathcal{P}_k\subset\Omega} \widehat{\lambda}^{\rm L}_k=
\widehat{\lambda}^{\rm L}_{R(1)}=\lambda^{\rm L}_{r(1)}=\dots=
\lambda^{\rm L}_{r(d)}.
\end{equation*}
Next, we remove $\varphi_{R(1)}$ from all $d$ bases approximating
the components of $\bm{u}=(u_1,\dots,u_d)$.
Then 
\begin{equation*}
\lambda_{2d}\ge\dots\lambda_{d+1}\ge 
\min_{\mathsf{v}\in \mathbb{R}^N,\, \mathsf{v}
\ne \mathsf{0},\, \mathsf{v}_{R(1)}=0, 
\dots,\mathsf{v}_{(d-1)N+R(1)}=0,   }
\frac{\mathsf{v}^T\mathsf{Cv}}{\mathsf{v}^T\widetilde{\mathsf{C}}\mathsf{v}}\ge \lambda^{\rm L}_{r(d+1)}=\dots=\lambda^{\rm L}_{r(2d)},
\end{equation*}
where the last inequality follows from
\begin{equation*}
\frac{\mathsf{v}^T\mathsf{Cv}}{\mathsf{v}^T\widetilde{\mathsf{C}}\mathsf{v}}
=\frac{\int_\mathcal{D} (\bm{\partial u})^T\cdot{\bm{C}}\bm{\partial u} \,{\rm d}\bm{x}}
{\int_\mathcal{D} (\bm{\partial u})^T\cdot\widetilde{\bm{C}}\bm{\partial u} \,{\rm d}\bm{x}}
\ge  \min_{\mathcal{P}_k\subset\mathcal{D}}
\widehat{\lambda}^{\rm L}_k=
\widehat{\lambda}^{\rm L}_{R(2)}= \lambda^{\rm L}_{r(d+1)}=\dots=
\lambda^{\rm L}_{r(2d)},
\end{equation*}
where $\mathsf{v}_{R(1)}=0,\dots,\mathsf{v}_{(d-1)N+R(1)}=0$,
and correspondingly,
\begin{equation*}
\mathcal{D}=\cup_{j\in\{1,\dots,N\}\setminus \{R(1)\}} \mathcal{P}_j.
\end{equation*}
Continuing further in this way, we can prove the lower bounds in~\eqref{thm1a}.
Analogously, we can get the upper bounds.
\end{proof}

\begin{ex}\label{ex_ela1}
Assume the elasticity equation with homogeneous Dirichlet 
boundary conditions,
$d=2$, $\Omega=(-\pi,\pi)^2$, $N=21^2$, and the data
\begin{equation}\label{Cex}
\bm{C}(\bm{x})=\frac{E(\bm{x})}{(1+\nu)(1-2\nu)}
\left(\begin{array}{ccc}
1-\nu   &  \nu & 0\\
\nu  &     1-\nu & 0\\
0 & 0  & 0.5\,(1-2\nu)
\end{array}\right),
\end{equation}
where 
\begin{equation*}
E(\bm{x})=1+0.3\,{\rm sign}\,(x_1x_2),\quad \nu=0.2.
\end{equation*}
Preconditioning is performed with the constant (homogeneous) 
data of the type~\eqref{Cex}
with $E=1$ and either $\nu=0$ or $\nu=0.2$, denoted by 
$\widetilde{\bm{C}}_1$ and $\widetilde{\bm{C}}_2$, respectively.
A~uniform grid with piece-wise bilinear FE functions is employed.
We can see in Figure~\ref{fff4} that the preconditioning matrix 
using the data $\widetilde{\bm{C}}_2$, which are closer to 
${\bm{C}}$, yields the spectrum of the preconditioned 
matrix closer to unity. 
Moreover, we can notice two clusters of eigenvalues approximately 
equal to 0.7 and 1.3, 
respectively.
The numbers of the CG steps to reduce the energy norms of the errors
by the factor of $10^{-9}$ are 14 and 11 for  
$\widetilde{\bm{C}}_1$ and $\widetilde{\bm{C}}_2$, respectively,
when we consider $\bm{F}=(1,0)^T$.
In this example, $\widetilde{\bm{C}}_1$
is diagonal, while $\widetilde{\bm{C}}_2$ is more filled in. Therefore,
the overall efficiency strongly depends on implementation
of the preconditioner.
These considerations are, however, behind the scope of this paper.
\begin{figure}[ht]
\vskip2pt plus 2pt \vbox to50pt{\vfil\hbox{\includegraphics{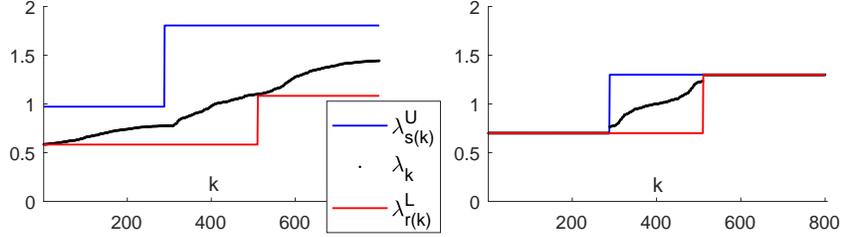}}}
\vspace{0.7in}
\caption{Lower ($\lambda^{\rm L}_{r(k)}$)
	and upper ($\lambda^{\rm U}_{s(k)}$) bounds on eigenvalues 
	$\lambda_{k}$ of
	the elasticity problem of Example~\ref{ex_ela1} with $N=21^2$ preconditioned by operators with 
	$\widetilde{\bm{C}}_1$ (left)
	and $\widetilde{\bm{C}}_2$ (right).}
\label{fff4} 
\end{figure}
\end{ex}

\begin{remark}
The bilinear form $(\bm{u},\bm{v})_C$ associated with the 
linear elasticity operator is equivalent with the following
bilinear forms defined in $V_0^d$,
see~\cite{Blaheta1994},
\begin{eqnarray}
(\bm{u},\bm{v})_{C,\triangle}&=&\int_{\Omega}\sum_{i,j=1}^d
\frac{\partial v_i}{\partial x_j}\frac{\partial u_i}{\partial x_j}
\,{\rm d}\bm{x}\nonumber\\
(\bm{u},\bm{v})_{C,\varepsilon}&=&\int_{\Omega}
(\bm{\partial v})^T\bm{\partial u}\,{\rm d}\bm{x}
\nonumber\\
(\bm{u},\bm{v})_{C,d}&=&\int_{\Omega}
(\bm{\partial}(v_1,0,\dots)^T)^T\bm{C\partial}(u_1,0,\dots)^T+\dots
\nonumber\\
&&\dots+(\bm{\partial}(\dots,0,v_d)^T)^T\bm{C\partial}(\dots,0,u_d)^T
\,{\rm d}\bm{x}.\nonumber
\end{eqnarray}
The equivalence constants and the proofs can be found in~\cite{Blaheta1994}
and in the references therein. 
We may notice that our preconditioning matrix
$\widetilde{\mathsf{C}}$ with the data in the form
$\widetilde{\bm{C}}(\bm{x})=\bm{I}$ is the same as 
the matrix of the discretized form $(\bm{u},\bm{v})_{C,\varepsilon}$.
Therefore, using our method for obtaining the bounds on the eigenvalues of preconditioned problems can be used to estimate
the equivalence constants of the above forms defined in finite-dimensional subspaces of $V_0^d$ spanned by the FE basis functions;
for example, we can immediately get
\begin{equation*}
\lambda^{\rm L}_{r(1)}(\bm{u},\bm{u})_{C,\varepsilon}\le 
(\bm{u},\bm{u})_{C}\le \lambda^{\rm U}_{s(dN)}(\bm{u},\bm{u})_{C,\varepsilon}.\end{equation*}
\end{remark}

\subsection{General remarks}\label{Sec33}

Let us now compare our results obtained for the diffusion equation
with the recent results from~\cite{G}. Analogies for the elasticity
equation can be considered straightforwardly.
In~\cite{G}, the existence of a pairing between 
the eigenvalues of the preconditioned matrix 
and the intervals obtained from the scalar data 
defined on the patches is proved. Especially, in any of the intervals,
some eigenvalue must be found.
This allows us to estimate the accuracy 
of the bounds provided that the scalar data 
are continuous or mildly changing in $\Omega$.
In our paper, instead, we get that
$\lambda_k\in [\lambda^{\rm L}_{r(k)},\lambda^{\rm U}_{s(k)}]$,
or $\lambda_{k}\in [\lambda^{\rm L}_{r(k-1)},\lambda^{\rm U}_{s(k)}]$
if the operator is semi-definite with the null space of the
dimension $1$. Let us note that 
\begin{equation*}
\lambda^{\rm L}_k\le \lambda^{\rm U}_k,\quad 
\lambda^{\rm L}_{r(k)}\le \lambda^{\rm U}_{s(k)},\quad
r(k)\le s(k),\quad k=1,\dots,N,
\end{equation*}
but $r(k)\ne s(k)$
in general, thus the intervals containing the 
individual eigenvalues are different than the intervals
obtained in~\cite{G}.
Sometimes, however, the intervals obtained by our method
and by the method of~\cite{G} (ordered appropriately) 
coincide; see the following example.
\begin{ex}\label{ex_G}
Let us consider the test problem from~\cite[Section~4]{G}:
the diffusion equation, $\Omega=(0,1)^2$, 
$\bm{A}(\bm{x})=\sin(x_1+x_2)\bm{I}$, and
homogeneous Dirichlet boundary conditions on $\partial\Omega$.
Let us use a uniform grid with piece-wise bilinear FE functions, $N=9^2$ or $N=19^2$. For preconditioning we use 
$\widetilde{\bm{A}}(\bm{x})=\bm{I}$.
The appropriatelly ordered bounds provided by~\cite{G} 
and the bounds obtained by our method coincide; they are 
displayed on Figure~\ref{fig_G2}.
\begin{figure}[ht]
\vskip2pt plus 2pt \vbox to50pt{\vfil\hbox{\includegraphics{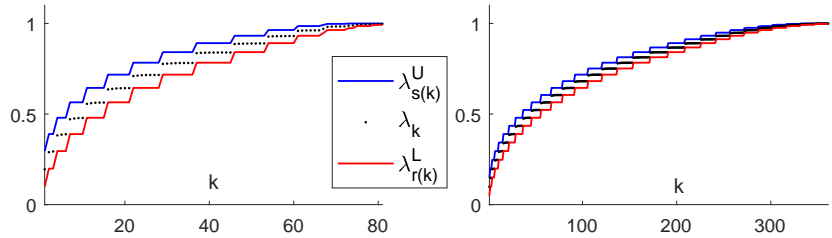}}}
\vspace{0.7in}
\caption{Lower ($\lambda^{\rm L}_{r(k)}$) and upper
	($\lambda^{\rm U}_{s(k)}$) bounds on eigenvalues $\lambda_{k}$ of 
	Example~\ref{ex_G} with $N=9^2$
	(left) and $N=19^2$ (right).}
\label{fig_G2} 
\end{figure}
\end{ex}

Modifying the approach developed in~\cite{G} to our setting
with the tensor data, we can also prove
that there exists a permutation $p:\{1,\dots,N\}\to \{1,\dots,N\}$
such that
\begin{equation}\label{jakoG}
\lambda_k\in [\lambda^{\rm L}_{p(k)},\lambda^{\rm U}_{p(k)}],\quad 
k=1,\dots,N.
\end{equation}
The only change in the proof consists of substituting the
extremes of the scalar material data on every patch $\mathcal{P}_j$ 
by the extremes of the eigenvalues of 
$\widetilde{\bm{A}}^{-1}(\bm{x})\bm{A}(\bm{x})$ on $\mathcal{P}_j$.
Therefore, we do not provide the proof here.

Using~\eqref{jakoG}, under some special conditions, 
analogously to the results of~\cite{G},
some eigenvalues can be identified exactly
including their multiplicity. Since we do not present the proof
of~\eqref{jakoG}, let us formulate and prove this statement
separately. For the sake of brevity, we
formulate it for the case of the nonsingular 
diffusion equation with the 
tensor data only.
Generalization to problems with vector valued unknowns is straightforward;
see also Example~\ref{ex_ela1}. 
\begin{lem}
Let there exist $c>0$ such that $\widetilde{\bm{A}}^{-1}(\bm{x})\bm{A}(\bm{x})=c\bm{I}$ 
on a union of $m$ patches $\mathcal{D}=\cup_{k=1}^m
\mathcal{P}_{j_k}$. Let none of the patches $\mathcal{P}_{j_k}$,
$k=1,\dots,m$, attaches to $\partial\Omega_2$
where $g_3(\bm{x})\ne 0$ or $\widetilde{g}_3\ne 0$,
and let the patches be
associated with $m$ linearly independent FE functions $\varphi_{j_1},\dots,\varphi_{j_m}$.
Let $\mathsf{A}$ be nonsingular.
Then $c$ is an eigenvalue of $\widetilde{\mathsf{A}}^{-1}\mathsf{A}$
of multiplicity at least $m$.
\end{lem}
\begin{proof}
Let $\mathsf{e}^{(j)}\in\mathbb{R}^N$ denote a zero vector with 
the $j$th component equal to unity.
Then for every $j=j_1,\dots,j_m$, 
\begin{equation*}
\frac{\mathsf{v}^T\mathsf{Ae}^{(j)}}
{\mathsf{v}^T\widetilde{\mathsf{A}}\mathsf{e}^{(j)}}=
\frac{\int_\Omega \nabla v\cdot\bm{A}\nabla\varphi_j
\,{\rm d}\bm{x}+\int_{\partial\Omega_2}g_3\varphi_j v\, {\rm d}S
}{\int_\Omega \nabla v\cdot\widetilde{\bm{A}}\nabla\varphi_j
\,{\rm d}\bm{x}+\int_{\partial\Omega_2}\widetilde{g}_3\varphi_j v\, {\rm d}S}=
\frac{c\int_\Omega \nabla v\cdot\widetilde{\bm{A}}\nabla\varphi_j
\,{\rm d}\bm{x}
}{\int_\Omega \nabla v\cdot\widetilde{\bm{A}}\nabla\varphi_j
\,{\rm d}\bm{x}}=c
\end{equation*}
for all $\mathsf{v}\in \mathbb{R}^N$, $\mathsf{v}\ne \mathsf{0}$. This means that $c$ is an eigenvalue of 
$\widetilde{\mathsf{A}}^{-1}{\mathsf{A}}$ associated with the eigenvectors $\mathsf{e}^{(j)}$, $j=j_1,\dots,j_m$.
Since the eigenvectors are linearly independent, the multiplicity of $c$ is at least $m$.
\end{proof}

Let us finally focus on limitations of our theory.
We could see that in some examples 
the bounds did not get closer to the real eigenvalues when
the mesh-size decreases. 
As a representative 2D example we can take
the diffusion equation with the constant data, say, 
$\bm{A}={\rm diag}\,(2,1)$
preconditioned by the Laplacian, i.e.~$\widetilde{\bm{A}}=
{\rm diag}\,(1,1)$. While the constant lower and upper 
bounds are obtained 
\begin{equation*}
\lambda^{\rm L}_k=1,\quad \lambda^{\rm U}_k=2,\quad k=1,\dots,N,
\end{equation*}
the true eigenvalues of $\widetilde{\mathsf{A}}^{-1}\mathsf{A}$ 
are distributed 
between these two bounds almost achieving both extremes 1 and 2.
We could conclude that if the data are of the tensor type 
and if the preconditioner is poor, i.e.~$\widetilde{\bm{A}}^{-1}
(\bm{x})\bm{A}(\bm{x})$ is not close enough to a multiple of the 
identity $\bm{I}$ in $\Omega$, the bounds 
$\lambda^{\rm L}_{r(k)}$ and $\lambda^{\rm U}_{s(k)}$
may not say much about the 
true eigenvalues; the types of the FE basis functions and  
of the mesh
influence the distribution of the true eigenvalues as well.

\section{Conclusion}

Up to our knowledge, \cite{G} is the first paper on estimating all
eigenvalues of a preconditioned discretized diffusion operator.
Motivated by~\cite{G}, we further contribute to
this theory by generalizing some of these results to 
vector valued equations with 
tensor data and with more general boundary
conditions preconditioned by arbitrary operators of the same type. Moreover, we provide guaranteed bounds (defined by~\eqref{LUbounds}
and by~\eqref{LUboundsC1}--\eqref{LUboundsC2}
for scalar and vector problems, respectively) to every particular 
eigenvalue. 
Analogously to~\cite{G},
the bounds are easily accessible and obtained solely from 
the data defined on supports of the FE basis functions.
If the data are element-wise constant,
only $O(N)$ arithmetic operations 
and sorting of two series of $N$ numbers must be performed.
The Courant--Fisher min-max theorem is used in
our approach. 
Although we applied our method to only two types of elliptic differential equations, we are convinced that the same approach can be 
used in a wide variety of elliptic problems.

{\small
{\em Martin Ladeck\' y}, Czech Technical University in Prague, Prague, Czech Republic, e-mail: \texttt{martin.ladecky@\allowbreak fsv.cvut.cz}.
{\em Ivana Pultarov\' a}, Czech Technical University in Prague, Prague, Czech Republic, e-mail: \texttt{ivana.pultarova@\allowbreak fsv.cvut.cz}.
{\em Jan Zeman}, Czech Technical University in Prague, Prague, Czech Republic, e-mail: \texttt{jan.zeman@\allowbreak fsv.cvut.cz}.
}

\end{document}